\documentclass[a4paper,11pt]{article}

\usepackage[english]{babel}
\usepackage{authblk,fullpage}
\usepackage{amsthm,amsmath,graphicx,url,pdfsync,enumerate,booktabs,xcolor,amssymb}        
\allowdisplaybreaks

\usepackage{graphicx}%
\usepackage{multirow}%
\usepackage{amsmath,amssymb,amsfonts}%
\usepackage{amsthm}%
\usepackage{mathrsfs}%
\usepackage[title]{appendix}%
\usepackage{xcolor}%
\usepackage{textcomp}%
\usepackage{manyfoot}%
\usepackage{booktabs}%
\usepackage{algorithm}%
\usepackage{algorithmicx}%
\usepackage{algpseudocode}%
\usepackage{listings}%
\usepackage{enumerate,tikz}%
\usepackage{thmtools}

\newtheorem{theorem}{Theorem}[section]
\newtheorem{definition}{Definition}[section]

\newtheorem{remark}{Remark}[section]

\newtheorem{example}{Example}[section]
\newtheorem{corollary}{Corollary}[section]

\newcommand{\R}{\mathbb{R}}

\title{\bf A Complete Characterization of the Inverse Eigenvector Centrality Problem \\
	for Undirected Graphs}

\author[1]{Mauro Passacantando}

\author[2]{Fabio Raciti}

\affil[1]{\small University of Milano-Bicocca, Department of Business and Law, Via Bicocca degli Arcimboldi 8, 20126 Milan, Italy, \texttt{mauro.passacantando@unimib.it}}

\affil[2]{\small Department of Mathematics and Computer Science, University of Catania, Viale A. Doria 6, 95125 Catania, Italy, \texttt{fabio.raciti@unict.it}}

\date{}

\begin{document}

\maketitle

\noindent\textbf{Abstract:} 
We study the inverse eigenvector centrality problem on connected undirected graphs, namely, whether a given positive vector can be realized by assigning suitable edge weights. 
We provide a complete characterization in terms of stable sets and their neighborhoods, showing that the undirected case requires nontrivial global constraints absent in the directed setting.

\


\


\maketitle

\section{Introduction}
\label{sec:1}

Measures of centrality play a fundamental role in network science, where they quantify the importance of nodes across a wide range of applications, from social and biological systems to technological networks. 
We refer to~\cite{Newman2018} for a comprehensive overview.

Among these, eigenvector centrality is one of the most widely used measures~\cite{Bonacich1972, Bonacich1987}. 
It assigns each node a score proportional to the sum of its neighbors' scores, thereby capturing the idea that connections to important nodes contribute more to a node's importance. 
In its standard formulation, eigenvector centrality is defined as the positive eigenvector associated with the largest eigenvalue of a (possibly weighted) adjacency matrix of a connected graph. 
The Perron-Frobenius theorem ensures that the largest eigenvalue is simple and admits a unique (up to scaling) strictly positive eigenvector, while all other eigenvectors necessarily have mixed signs~\cite{HornJohnson2013,Minc1988}.
As a well-established measure of node importance that accounts for the influence of neighboring nodes, eigenvector centrality continues to find applications in complex networks, including tennis competition systems and agro-product circulation networks~\cite{Arcagni2023, Ma2025}.

In this paper, we consider the corresponding inverse problem, namely, whether a given positive vector can be realized as an eigenvector centrality by suitably choosing the edge weights of a graph. 
This problem has been investigated in the literature; see, for instance,~\cite{Nicosia2012}, where it is shown that the inverse problem admits infinitely many positive solutions for a directed, strongly connected graph. 
This nonuniqueness has been exploited very recently in~\cite{PassacantandoRaciti2026} to introduce a family of six optimization problems defined over the solution set of the inverse eigenvector centrality problem, each corresponding to a different strategy for realizing the same target centrality. 
Recent work has also highlighted the sensitivity of eigenvector centrality to small perturbations of the edge weights. 
In particular,~\cite{BenziGuglielmi2025} show that even small structured perturbations can significantly alter the induced ranking.

The inverse problem is markedly different for undirected graphs. In this case, the feasibility of a given vector as an eigenvector centrality cannot be ensured by local balance conditions alone, but instead depends on global structural constraints of the graph. We show that these constraints can be expressed in terms of stable sets and their external neighborhoods, leading to a complete characterization of the inverse problem.

Finally, we remark that our study is related, at a broader level, to the classical inverse eigenvalue problem, which concerns the construction of matrices with prescribed spectral properties (see, e.g.,~\cite{Chu1998,Friedland1989}). 
However, the inverse eigenvector centrality problem exhibits distinctive combinatorial features that are specific to network structures.
Our main contribution is a complete characterization of the inverse eigenvector centrality problem on connected undirected graphs. We prove that a positive vector can be realized as an eigenvector centrality if and only if it satisfies a family of equalities and (strict) inequalities indexed by stable sets and their external neighborhoods (see, e.g.,~\cite{michini2013}). 
The result is obtained via a convex-analytic reformulation of the problem, combined with structural properties of stable sets and an application of Farkas Lemma.

\section{The Inverse Eigenvector Centrality Problem}
\label{sec2}

Let $G = (V,E)$ be a connected simple undirected graph, where $n=|V|$ and $m=|E|$, with a positive weight $w_{ij}$ associated with each edge $\{i,j\} \in E$. 
Then, the weighted adjacency matrix $A \in \R^{n \times n}$, defined as follows
\begin{align}\label{e:A}
	a_{ij} =
	\begin{cases}
	w_{ij} & \text{if $\{i,j\} \in E$},
	\\
	0 & \text{otherwise},	
	\end{cases}
\end{align}
is irreducible~\cite{HornJohnson2013}. 
Moreover, the Perron-Frobenius theorem guarantees that the spectral radius $\rho>0$ of $A$ is an algebraically simple eigenvalue, the corresponding eigenvector $c$ is unique (up to scaling), and it is the unique positive eigenvector of $A$ (see, e.g.,~\cite{Minc1988}).
Each component $c_i$ is called the eigenvector centrality of node $i\in V$ and is proportional to the sum of the centralities of its neighbors:
\[
	\rho \, c_i = \sum_{j \in V} a_{ij} \, c_j,
\]
that is, a node is considered important if it is connected to other important nodes~\cite{Bonacich1972}.

We now consider the inverse eigenvector centrality problem.
 
\begin{definition}\label{def1}
Given a positive vector $c \in \R^n$ and $\rho>0$, the inverse eigenvector centrality problem (IECP) consists in finding positive edge weights $w_{ij}$ such that $A c = \rho \, c$.
\end{definition}

We note that in the above definition, the parameter $\rho$ can always be assumed to be equal to 1 without loss of generality. 
Moreover, the requirement that all edge weights be positive and not only nonnegative is crucial for $c$ to actually be an eigenvector centrality vector, as the following example shows.

\begin{example}[label=ex1]
Consider the simple graph shown in Figure~\ref{fig:ex}. 
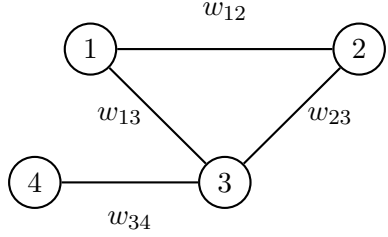
\begin{figure}[htbp]
	\centering
	\begin{tikzpicture}[
		node distance=2.5cm,
		every node/.style={circle, draw, thick, minimum size=0.5cm},
		edge/.style={draw, thick}
		]
		
		\node (3) [fill=gray!0] {3};
		\node (1) [above left of=3] {1};
		\node (2) [above right of=3] {2};
		\node (4) [left of=3] {4};
		
		\draw[edge] (1) -- (3) node[midway, left, draw=none] {$w_{13}$};
		\draw[edge] (2) -- (3) node[midway, right, draw=none] {$w_{23}$};
		\draw[edge] (4) -- (3) node[midway, below, draw=none] {$w_{34}$};
		\draw[edge] (1) -- (2) node[midway, above, draw=none] {$w_{12}$};
		
	\end{tikzpicture}
	\caption{Graph of Example~\ref{ex1}.}
	\label{fig:ex}
\end{figure}
If we fix $c = (1, 1, 1, 1)$ and $\rho=1$, then the system $A c = \rho \, c$ reads
\begin{align*}
	\begin{cases}
	w_{12} + w_{13} = 1 \\
	w_{12} + w_{23} = 1 \\
	w_{13} + w_{23} + w_{34} = 1 \\
	w_{34} = 1
	\end{cases} 
\end{align*}
and its unique nonnegative solution is $w_{12}=w_{34}=1$, $w_{13}=w_{23}=0$. 
The adjacency matrix corresponding to this set of weights is
$$
A = \begin{pmatrix}
0 & 1 & 0 & 0 \\
1 & 0 & 0 & 0 \\
0 & 0 & 0 & 1 \\
0 & 0 & 1 & 0
\end{pmatrix},
$$
its largest eigenvalue is 1 with multiplicity 2, and the corresponding eigenvectors are of the form $(a,a,b,b)$, with $a,b \in \R$. 
Therefore, forcing some weights to zero effectively disconnects the graph, leading to the loss of irreducibility of the matrix $A$. 
Thus, $(1, 1, 1, 1)$ is not the unique (positive) eigenvector associated with the eigenvalue 1, and thus cannot be considered a ``true'' centrality.
\end{example}

We now provide a set of necessary and sufficient conditions on the vector $c$ for the inverse eigenvector centrality problem to admit a solution. For any subset $S \subset V$, we denote its external neighborhood as
$$
N(S) := \left\{j \in V \setminus S:\ \exists\ i \in S \text{ such that } \{i,j\} \in E \right\}.
$$
Moreover, we recall that $S \subset V$ is called a stable set if for any $i,j \in S$ one has $\{i,j\} \notin E$. 
	
\begin{theorem}\label{t1}
Let $G = (V, E)$ be a connected simple undirected graph and $c$ a positive vector in $\R^n$. 
We define the following families of stable sets:
\begin{align*}
& \mathcal{S}_1 :=\{ S \subset V:\ 
\text{$S \neq \emptyset$, $S$ is stable and there is no edge $\{k,\ell\} \in E$ s.t. $k,\ell \notin S$}\},
\\
& \mathcal{S}_2:=\{ S \subset V:\ 
\text{$S \neq \emptyset$, $S$ is stable and there exists an edge $\{k,\ell\} \in E$ s.t. $k,\ell \notin S$}\}.
\end{align*}
Then, IECP admits a solution, i.e., $c$ is an eigenvector centrality vector, if and only if the following conditions hold:
\begin{align}
\sum_{j \in S} c^2_j = \sum_{j \in N(S)} c^2_j
\qquad
\text{for any $S \in \mathcal{S}_1$},
\label{e:NSC1a}
\\	 
\sum_{j \in S} c^2_j < \sum_{j \in N(S)} c^2_j
\qquad
\text{for any $S \in \mathcal{S}_2$}.
\label{e:NSC1b}
\end{align}
\end{theorem}

\begin{proof}
First, we prove that~\eqref{e:NSC1a}--\eqref{e:NSC1b} are necessary conditions. 
If $c$ is an eigenvector centrality vector, then there exists a symmetric weighted adjacency matrix $A$, defined as in~\eqref{e:A}, such that
$$
	c_j = \sum_{i \in V} a_{ij} c_i
	\qquad \forall\ j \in V.
$$
Multiplying each equation by $c_j$, we get
\begin{align}\label{e:NSC22}
	c^2_j = \sum_{i \in V} a_{ij} c_i c_j
\qquad \forall\ j \in V.
\end{align}
For any stable set $S \subset V$, we have
\begin{align}\label{e:NSC33}
\sum_{j \in S} c^2_j 
= \sum_{j \in S} \sum_{i \in V} a_{ij} c_i c_j
= \sum_{j \in S} \sum_{i \in N(S)} w_{ij} c_i c_j
= \sum_{i \in N(S)} \sum_{j \in S} w_{ji} c_i c_j.
\end{align}
If $S \in \mathcal{S}_1$, then
\begin{align}\label{e:NSC44}
\sum_{i \in N(S)} \sum_{j \in S} w_{ji} c_i c_j
=
\sum_{i \in N(S)} \sum_{j \in V} a_{ji} c_i c_j
= 
\sum_{i \in N(S)} c^2_i,
\end{align}
hence~\eqref{e:NSC33} and \eqref{e:NSC44} imply \eqref{e:NSC1a}.
If $S \in \mathcal{S}_2$, then there exists an edge $\{k,\ell\} \in E$ such that $k \notin S$ and $\ell \notin S$. Since the graph is connected, we can assume without loss of generality that $k \in N(S)$, thus
\begin{align*}
	\sum_{j \in S} a_{jk} c_k c_j
	<
	\sum_{j \in S} a_{jk} c_k c_j + w_{\ell k} c_k c_\ell
	\leq	
	\sum_{j \in V} a_{jk} c_k c_j.
\end{align*}
Therefore, we have
\begin{align}\label{e:NSC55}
	\sum_{i \in N(S)} \sum_{j \in S} a_{ji} c_i c_j
	<
	\sum_{i \in N(S)} \sum_{j \in V} a_{ji} c_i c_j
	= 
	\sum_{i \in N(S)} c^2_i,
\end{align}
hence~\eqref{e:NSC33} and \eqref{e:NSC55} imply \eqref{e:NSC1b}.

We now prove that~\eqref{e:NSC1a}--\eqref{e:NSC1b} are sufficient conditions. 
The vector $c$ is an eigenvector centrality vector if and only if there exists $\varepsilon >0$ such that the system 
\begin{align}\label{e:NSC3}
	\begin{cases}
	\sum\limits_{i \in N(j)} w_{ij} c_i = c_j 
	& \qquad \forall\ j \in V,
	\\
	w_{ij} \geq \varepsilon 
	& \qquad \forall\ \{i,j\} \in E,
	\end{cases}
\end{align}
admits at least a solution.
The system~\eqref{e:NSC3} is equivalent to 
\begin{align*}
	\begin{cases}
	\sum\limits_{i \in N(j)} w_{ij} c_i c_j = c^2_j 
	& \qquad \forall\ j \in V,
	\\
	w_{ij} \geq \varepsilon 
	& \qquad \forall\ \{i,j\} \in E,
	\end{cases}
\end{align*}
which, in turn, can be rewritten as
\begin{align*}
	\begin{cases}
	\sum\limits_{i \in N(j)} z_{ij} c_i c_j 
	= 
	c^2_j - \varepsilon \sum\limits_{i \in N(j)} c_i c_j 
	& \qquad \forall\ j \in V,
	\\
	z_{ij} \geq 0  
	& \qquad \forall\ \{i,j\} \in E,
	\end{cases}
\end{align*}
where $z_{ij}=w_{ij} - \varepsilon$ for any $\{i,j\} \in E$. 
The latter system can be written in matrix form as 
\begin{align}\label{e:NSC5}
\begin{cases}
	B z = q
	\\
	z \geq 0
\end{cases}
\end{align}
where $B \in \R^{n \times m}$ is defined as
$$
	B_{k, \{i,j\}} = 
	\begin{cases}
	c_i c_j & \text{if $k=i$ or $k=j$},
	\\
	0 & \text{otherwise},
	\end{cases}
$$ 
and $q_j = c^2_j - \varepsilon \sum\limits_{i \in N(j)} c_i c_j$ for any $j \in V$.
By the Farkas Lemma, the system~\eqref{e:NSC5} admits a solution if and only if the  system 
$$
\begin{cases}
	q^\top x >0
	\\
	B^\top x \leq 0
\end{cases}
$$
is impossible, i.e., 
\begin{align}\label{e:NSC6}
	q^\top x \leq 0 \qquad \forall\ x \in C:=\{ x \in \R^n:\ B^\top x \leq 0\}.
\end{align}
Since the polyhedral cone $C$ has a finite set of extreme rays $v^1,\dots,v^p$, then  $c$ is an eigenvector centrality vector if and only if there exists $\varepsilon >0$ such that
\begin{align}\label{e:NSC7}
	q^\top v^i \leq 0 \qquad \forall\ i \in \{1,\dots,p\}.
\end{align}
We remark that $C$ can be written as $C=\{x \in \R^n:\ x_i + x_j \leq 0 \quad \forall\ \{i,j\} \in E\}$. 
If we define the polytope $P = \{ x \in C:\ x_i \geq -1 \}$, then the extreme rays of $C$ are the vertices of $P$ that are different from the null vector. 
If we consider the map $M: \R^n \to \R^n$ defined as $x_i \mapsto y_i = (x_i+1)/2$ for any $i=1,\dots,n$, then the image of $P$ through $M$ is the fractional stable set polytope
$$
	FSTAB = \{ y \in \R^n:\ y_i \geq 0, \quad y_i + y_j \leq 1 \quad \forall\ \{i,j\} \in E\},
$$ 
and the vertices of $P$ and $FSTAB$ are in a one-to-one correspondence.
The well-known result by~\cite{NemTro74} guarantees that the vertices of $FSTAB$ are the vectors $y \in \{0, 1/2, 1\}^n$ such that the following conditions hold:
\begin{itemize}
	\item $\{i \in V:\ y_i=1\}$ is a stable set $S$;
	\item $\{i \in V:\ y_i=0\}$ contains the set $N(S)$ of neighbors of $S$;
	\item Each connected component of the subgraph induced by the subset $\{i \in V:\ y_i=1/2\}$ contains an odd cycle.
\end{itemize}
Therefore, the vertices of $P$ are the vectors $x \in \{ -1, 0, 1\}^n$ such that the following conditions hold:
\begin{itemize}
	\item $\{i \in V:\ x_i=1\}$ is a stable set $S$;
	\item $\{i \in V:\ x_i=-1\}$ contains the set $N(S)$ of neighbors of $S$;
	\item Each connected component of the subgraph induced by the subset $\{i \in V:\ x_i=0\}$ contains an odd cycle.
\end{itemize}
For clarity, let divide the vertices of $P$ into three subsets: vectors that do not have any components equal to 1, vectors that do not have any components equal to 0, and all remaining vectors. Hence, the system of inequalities~\eqref{e:NSC7} can be written as
\begin{align}
q^\top v^i \leq 0 & \qquad \forall\ i \in \{1,\dots,p\}: v^i \in \{-1,0\}^n, 
\label{e:NSC7a}
\\
q^\top v^i \leq 0 & \qquad \forall\ i \in \{1,\dots,p\}: v^i \in \{-1,1\}^n, 
\label{e:NSC7b}
\\
q^\top v^i \leq 0 & \qquad \forall\ i \in \{1,\dots,p\}: v^i \in \{-1,0,1\}^n. 
\label{e:NSC7c}
\end{align}
We remark that if $\varepsilon$ is small enough, then $q>0$, hence conditions~\eqref{e:NSC7a} are always satisfied as $v^i \leq 0$. 
Conditions~\eqref{e:NSC7b} are equivalent to the system
\begin{align*}
\sum_{j \in S} q_j \leq \sum_{j \in V \setminus S} q_j
\qquad \forall\ \text{stable set $S \subset V$},
\end{align*}
that is equivalent to
\begin{align}\label{e:NSC8}
	\sum_{j \in S} c^2_j 
	+ \varepsilon \left[
	\sum_{j \in V \setminus S} \sum_{i \in N(j)} c_i c_j - 
	\sum_{j \in S} \sum_{i \in N(j)} c_i c_j 
	\right]
	\leq 
	\sum_{j \in V \setminus S} c^2_j
	\qquad \forall\ \text{stable set $S \subset V$}.
\end{align}
We remark that
\begin{align*}
\sum_{j \in V \setminus S} \sum_{i \in N(j)} c_i c_j - 
\sum_{j \in S} \sum_{i \in N(j)} c_i c_j 
& =
\sum_{\substack{\{i,j\} \in E: \\ i \in S, j \notin S}} c_ic_j
+
2 \sum_{\substack{\{i,j\} \in E: \\ i \notin S, j \notin S}} c_ic_j
	-
\sum_{\substack{\{i,j\} \in E: \\ i \in S, j \notin S}} c_ic_j	
\\
& = 2 \sum_{\substack{\{i,j\} \in E: \\ i \notin S, j \notin S}} c_ic_j.
\end{align*}
If $S \in \mathcal{S}_1$, then the latter sum is null and \eqref{e:NSC1a} implies
$$
	\sum_{j \in S} c^2_j 
	=
	\sum_{j \in N(S)} c^2_j 
	\leq
	\sum_{j \in V \setminus S} c^2_j,
$$
hence~\eqref{e:NSC8} holds. 
If $S \in \mathcal{S}_2$, then \eqref{e:NSC1b} implies
$$
\sum_{j \in S} c^2_j 
<
\sum_{j \in N(S)} c^2_j 
\leq
\sum_{j \in V \setminus S} c^2_j,
$$
hence~\eqref{e:NSC8} holds provided that $\varepsilon$ is small enough.

Conditions~\eqref{e:NSC7c} read
\begin{align}\label{e:NSC9}
	\sum_{i \in S} q_i \leq \sum_{i \in T} q_i
\end{align}
holds for any stable set $S \subset V$ and any set $T \subseteq V \setminus S$ such that $N(S) \subseteq T$ and each connected component of $V \setminus (S \cup T)$ contains an odd cycle.
The inequality~\eqref{e:NSC9} reads
\begin{align}\label{e:NSC10}
	\sum_{j \in S} c^2_j 
	+ \varepsilon \left[
	\sum_{j \in T} \sum_{i \in N(j)} c_i c_j - 
	\sum_{j \in S} \sum_{i \in N(j)} c_i c_j 
	\right]
	\leq 
	\sum_{j \in T} c^2_j.
\end{align}
We have
\begin{align*}
	\sum_{j \in T} \sum_{i \in N(j)} c_i c_j - 
	\sum_{j \in S} \sum_{i \in N(j)} c_i c_j 
	& =
	\sum_{\substack{\{i,j\} \in E: \\ i \in S, j \notin S}} c_ic_j
	+
	\sum_{j \in T} \sum_{\substack{i \in N(j): \\ i \notin S}} c_ic_j
	-
	\sum_{\substack{\{i,j\} \in E: \\ i \in S, j \notin S}} c_ic_j	
	\\
	& = \sum_{j \in T} \sum_{\substack{i \in N(j): \\ i \notin S}} c_ic_j.
\end{align*}
If $S \in \mathcal{S}_1$, then the latter sum is null and \eqref{e:NSC1a} implies
$$
\sum_{j \in S} c^2_j 
=
\sum_{j \in N(S)} c^2_j 
\leq
\sum_{j \in T} c^2_j,
$$
hence~\eqref{e:NSC10} holds. 
If $S \in \mathcal{S}_2$, then \eqref{e:NSC1b} implies
$$
\sum_{j \in S} c^2_j 
<
\sum_{j \in N(S)} c^2_j 
\leq
\sum_{j \in T} c^2_j,
$$
hence~\eqref{e:NSC10} holds provided that $\varepsilon$ is small enough.

Therefore, we proved that there exists $\varepsilon>0$ small enough such that the system~\eqref{e:NSC7a}--\eqref{e:NSC7c} is satisfied. Thus, $c$ is an eigenvector centrality vector.
\end{proof}

\begin{remark}\label{rem1}
	Some inequalities of system~\eqref{e:NSC1a}--\eqref{e:NSC1b} can be dropped as redundant. 
	Indeed, system~\eqref{e:NSC1b} can be reduced to 
	$$
	\sum_{j \in S} c^2_j < \sum_{j \in N(S)} c^2_j
	\qquad
	\forall\ S \in \mathcal{F},
	$$
	where $\mathcal{F}$ is the family of stable sets $S \in \mathcal{S}_2$ that satisfy the following conditions:
	\begin{enumerate}[a)]
		\item there is no partition of $S$ as $S = S' \cup S''$ such that $N(S') \cap N(S'') = \emptyset$;
		
		\item there is no stable set $S'$ such that $S \subset S'$ and $N(S)=N(S')$.
	\end{enumerate}
	Indeed, if $S \in \mathcal{S}_2$ and there exists a partition of $S$ into two stable sets $S'$ and $S''$ such that $N(S') \cap N(S'') = \emptyset$, then, $S', S'' \in \mathcal{S}_2$ and the inequality corresponding to $S$ reads
	$$
	\sum_{j \in S'} c^2_j + \sum_{j \in S''} c^2_j
	=
	\sum_{j \in S} c^2_j 
	< 
	\sum_{j \in N(S)} c^2_j
	=
	\sum_{j \in N(S')} c^2_j + \sum_{j \in N(S'')} c^2_j,
	$$
	hence it is redundant as it is the sum of the inequalities corresponding to $S'$ and $S''$.
	Moreover, if there exists a stable set $S'$ such that $S \subset S'$ and $N(S)=N(S')$, then 
	$$
	\sum_{j \in S} c^2_j
	<
	\sum_{j \in S'} c^2_j
	\leq 
	\sum_{j \in N(S')} c^2_j
	=
	\sum_{j \in N(S)} c^2_j,
	$$
	hence the inequality corresponding to $S$ is redundant as it is implied by the condition corresponding to $S'$, which is an equality if $S' \in \mathcal{S}_1$ or a strict inequality if $S' \in \mathcal{S}_2$. 
\end{remark}

\begin{example}[continues=ex1]
We apply Theorem~\ref{t1} and Remark~\ref{rem1} to the graph of Example~\ref{ex1}. Given a positive vector $c \in \R^4$, the system $A c = c$ reads	
$$
\begin{cases}
			c_2 w_{12} + c_3 w_{13} = c_1 \\
			c_1 w_{12} + c_3 w_{23} = c_2 \\
			c_1 w_{13} + c_2 w_{23} + c_4 w_{34} = c_3 \\
			c_3 w_{34} = c_4
		\end{cases} 
$$
whose unique solution is
$$
\begin{cases}
w_{12} = \dfrac{c_1^2 + c_2^2 + c_4^2 - c_3^2}{2 c_1 c_2}, 
&
w_{13} = \dfrac{c_1^2 + c_3^2 - c_2^2 - c_4^2}{2 c_1 c_3}, 
\\[4mm]
w_{23} = \dfrac{c_2^2 + c_3^2 - c_1^2 - c_4^2}{2 c_2 c_3}, 
&
w_{34} = \dfrac{c_4}{c_3}. 
\end{cases} 
$$
This solution is positive if and only if the following three conditions hold:
\begin{equation}\label{ex1:cond}
\begin{cases}
c_3^2 < c_1^2 + c_2^2 + c_4^2 \\
c_2^2 + c_4^2 < c_1^2 + c_3^2 \\
c_1^2 + c_4^2 < c_2^2 + c_3^2 
\end{cases} 
\end{equation}
On the other hand, the stable sets of the considered graph are $\{1\}, \{2\}, \{3\}, \{4\}, \{1,4\}, \{2,4\}$, which all belong to the family $\mathcal{S}_2$. Moreover, the family $\mathcal{F}$ defined in Remark~\ref{rem1} only contains the sets $\{3\}, \{4\}, \{1,4\}, \{2,4\}$. Hence, by Theorem~\ref{t1}, IECP admits a solution if and only if the following conditions hold:
\begin{align}\label{ex1:cond2}
\begin{cases}
	c_3^2 < c_1^2 + c_2^2 + c_4^2 \\
	c_4^2 < c_3^2 \\
	c_1^2 + c_4^2 < c_2^2 + c_3^2 \\
	c_2^2 + c_4^2 < c_1^2 + c_3^2 
\end{cases} 
\end{align}
The second inequality of the latter system is redundant as it is the sum of the third and fourth inequalities, thus system~\eqref{ex1:cond2} actually coincides with~\eqref{ex1:cond}.
\end{example}

The following result shows that the conditions~\eqref{e:NSC1a}-\eqref{e:NSC1b} simplify drastically for graphs with a special structure.

\begin{corollary}
Let $G = (V, E)$ be a connected simple undirected graph and $c \in \R^n$ a positive vector. 
\begin{enumerate}[a)]

\item If $G$ is complete, then IECP admits a solution if and only if 
\begin{align}\label{e:complete}
2 \max\limits_{j \in V} c^2_j < \sum\limits_{j \in V} c^2_j.
\end{align}

\item If $G$ is a complete bipartite graph, where $V=V_1 \cup V_2$, with $V_1 \cap V_2 = \emptyset$ and \linebreak $E=\{\{i,j\}:\ i \in V_1, \ j \in V_2\}$, then IECP admits a solution if and only if 
$$
\sum\limits_{j \in V_1} c^2_j = \sum\limits_{j \in V_2} c^2_j.
$$

\item If $G$ is a star graph with internal node 1, then IECP admits a solution if and only if 
$$
c^2_1 = \sum\limits_{j \in V \setminus \{1\}} c^2_j.
$$

\item If $G$ is a chain graph with $V=\{1,\dots,n\}$ and $E=\{ \{i, i+1 \}:\ i \in V \setminus \{n\}\}$, then IECP admits a solution if and only if the following chain of inequalities and equality holds:
\begin{align}\label{e:chain1}
c_1^2 < c_2^2 < c_1^2 + c_3^2 < c_2^2 + c_4^2 < c_1^2 + c_3^2 + c_5^2 < \dots < \sum_{i \in V_{\rm odd}} c_i^2 = \sum_{i \in V_{\rm even}} c_i^2,
\end{align}
where $V_{\rm odd}=\{ i \in V:\ i \text{ is odd}\}$ and $V_{\rm even}=\{ i \in V:\ i \text{ is even}\}$.

\end{enumerate}	
\end{corollary}

\begin{proof} \
\begin{enumerate}[a)]

\item The only stable sets are the singletons $S=\{i\}$ for any $i \in V$, and they belong to the class $\mathcal{S}_2$. 
Then, condition~\eqref{e:NSC1b} reads
$$
c^2_i < \sum_{j \in V \setminus\{i\}} c^2_j
\qquad
\forall\ i \in V,
$$
that is equivalent to 
$$
2 c^2_i < \sum_{j \in V} c^2_j
\qquad
\forall\ i \in V,
$$
i.e., $2 \max\limits_{j \in V} c^2_j < \sum\limits_{j \in V} c^2_j$.

\item The only stable sets in the class $\mathcal{S}_1$ are $V_1$ and $V_2$, that provide the equality
\begin{align}\label{e:bipar}
	\sum_{j \in V_1} c^2_j = \sum_{j \in V_2} c^2_j.
\end{align}
Moreover, the stable sets in the class $\mathcal{S}_2$ are the proper subsets of either $V_1$ or $V_2$. If $S \subset V_1$, then the inequality
$$
\sum_{j \in S} c^2_j 
<
\sum_{j \in V_2} c^2_j
$$
is redundant as it is implied by~\eqref{e:bipar}. 
Similarly, it can be shown that the inequalities associated with the subsets of $V_2$ are redundant. 

\item It follows from b), since any star graph is a special case of a complete bipartite graph.

\item The only stable sets in class $\mathcal{S}_1$ are $V_{\rm odd}$ and $V_{\rm even}$, so the system~\eqref{e:NSC1a} reads
\begin{align*}
	\sum_{i \in V_{\rm odd}} c_i^2 = \sum_{i \in V_{\rm even}} c_i^2.
\end{align*}
The inequalities in~\eqref{e:chain1} follow from~\eqref{e:NSC1b}. 
To prove that conditions~\eqref{e:chain1} are equivalent to system~\eqref{e:NSC1a}--\eqref{e:NSC1b}, we note that the stable sets of the class $\mathcal{F}$ described in Remark~\ref{rem1} are sets of consecutive odd nodes or sets of consecutive even nodes. 

First, we consider stable sets of $\mathcal{F}$ that are singletons. 
If $S=\{i\}$ with $3 \leq i \leq n-1$ and $i$ is odd, then the following inequalities from~\eqref{e:chain1}:
$$
	\sum_{\substack{j=2 \\ \text{$j$ even}}}^{i-3} c_j^2
	<
	\sum_{\substack{j=1 \\ \text{$j$ odd}}}^{i-2} c_j^2,
	\qquad
	\sum_{\substack{j=1 \\ \text{$j$ odd}}}^{i} c_j^2
	<
	\sum_{\substack{j=2 \\ \text{$j$ even}}}^{i+1} c_j^2,
$$
imply that
$$
c_i^2 + \sum_{\substack{j=2 \\ \text{$j$ even}}}^{i-3} c_j^2
<
c_i^2 + \sum_{\substack{j=1 \\ \text{$j$ odd}}}^{i-2} c_j^2
=
c_i^2 + \sum_{\substack{j=1 \\ \text{$j$ odd}}}^{i} c_j^2
<
\sum_{\substack{j=2 \\ \text{$j$ even}}}^{i+1} c_j^2
=
\sum_{\substack{j=2 \\ \text{$j$ even}}}^{i-3} c_j^2 + c_{i-1}^2 + c_{i+1}^2,
$$
hence we get the inequality $c_i^2 < c_{i-1}^2 + c_{i+1}^2$ of~\eqref{e:NSC1b} corresponding to $S$. 
If $i=n$ and $n$ is odd, then the following inequality and equality from~\eqref{e:chain1}:
$$
\sum_{\substack{j=2 \\ \text{$j$ even}}}^{n-3} c_j^2
<
\sum_{\substack{j=1 \\ \text{$j$ odd}}}^{n-2} c_j^2,
\qquad
\sum_{\substack{j=1 \\ \text{$j$ odd}}}^{n} c_j^2
=
\sum_{\substack{j=2 \\ \text{$j$ even}}}^{n-1} c_j^2,
$$
imply that
$$
c_n^2 + \sum_{\substack{j=2 \\ \text{$j$ even}}}^{n-3} c_j^2
<
c_n^2 + \sum_{\substack{j=1 \\ \text{$j$ odd}}}^{n-2} c_j^2
=
c_i^2 + \sum_{\substack{j=1 \\ \text{$j$ odd}}}^{n} c_j^2
=
\sum_{\substack{j=2 \\ \text{$j$ even}}}^{n-1} c_j^2
=
\sum_{\substack{j=2 \\ \text{$j$ even}}}^{n-3} c_j^2 + c_{n-1}^2,
$$
hence we get $c_n^2 < c_{n-1}^2$.
If $3 \leq i \leq n$ and $i$ is even, then a similar argument can be used to prove that the inequality of~\eqref{e:NSC1b} corresponding to $S$ follows from~\eqref{e:chain1}.

If $S$ is a stable set of $\mathcal{F}$ consisting of two consecutive odd nodes, then we consider $S=\{i, i+2\}$ where $i$ is odd and $3 \leq i \leq n-3$. 
Then the following inequalities from~\eqref{e:chain1}:
$$
\sum_{\substack{j=2 \\ \text{$j$ even}}}^{i-3} c_j^2
<
\sum_{\substack{j=1 \\ \text{$j$ odd}}}^{i-2} c_j^2,
\qquad
\sum_{\substack{j=1 \\ \text{$j$ odd}}}^{i+2} c_j^2
<
\sum_{\substack{j=2 \\ \text{$j$ even}}}^{i+3} c_j^2,
$$
imply that
\begin{align*}
c_i^2 + c_{i+2}^2 + \sum_{\substack{j=2 \\ \text{$j$ even}}}^{i-3} c_j^2
& < c_i^2 + c_{i+2}^2 + \sum_{\substack{j=1 \\ \text{$j$ odd}}}^{i-2} c_j^2
\\
& = \sum_{\substack{j=1 \\ \text{$j$ odd}}}^{i+2} c_j^2
< \sum_{\substack{j=2 \\ \text{$j$ even}}}^{i+3} c_j^2
= c_{i-1}^2 + c_{i+1}^2 + c_{i+3}^2 + \sum_{\substack{j=2 \\ \text{$j$ even}}}^{i-3} c_j^2,
\end{align*}
hence we get the inequality $c_i^2 + c_{i+2}^2 < c_{i-1}^2 + c_{i+1}^2 + c_{i+3}^2$ of~\eqref{e:NSC1b} corresponding to $S$. 
Similarly, we can prove that the inequality $c_{n-2}^2 + c_{n}^2 < c_{n-3}^2 + c_{n-1}^2$ of~\eqref{e:NSC1b} follows from~\eqref{e:chain1}. 
The above argument can be extended to show that every inequality in~\eqref{e:NSC1b} corresponding to a stable set of the class $\mathcal{F}$ is implied by conditions~\eqref{e:chain1}, which are therefore sufficient for the IECP to admit a solution.
\end{enumerate}	
\end{proof}

\begin{remark}
Condition~\eqref{e:complete} for complete graphs is related to the problem of realizing prescribed positive row sums by a symmetric matrix with zero diagonal, i.e., a weighted analog of the degree-realization problem. 
The classical realizability results and constructive procedures in this direction go back to Hakimi~\cite{Hakimi1962}, who studied realizability of prescribed degrees and provided criteria/algorithms (Havel-Hakimi-type) for complete-graph realizations. 
\end{remark}

\section{Conclusions}

In this paper, we provided a complete characterization of the inverse eigenvector centrality problem for connected undirected graphs, showing that feasibility is governed by global constraints involving stable sets and their neighborhoods. 
These conditions fully determine whether a positive vector can be realized as an eigenvector centrality. While the characterization is complete from a theoretical standpoint, its direct application in practice may be challenging due to the combinatorial nature of the conditions.

A natural direction for future research concerns approximation: when a target centrality vector is not realizable, it is relevant to ask whether there exists a 'closest' realizable vector and how closeness should be defined. 
Possible approaches include minimizing distances between the target and feasible vectors, or preserving the induced ranking as much as possible. 
Investigating the existence, uniqueness, and efficient computation of such approximations represents an interesting avenue for further study.

\subsection*{Acknowledgements}

The research of F.R. was supported by the MUR research
programs PRIN2022 founded by the European Union - Next Generation EU (Project ``ACHILLES, eco-sustAinable effiCient tecHdrIven Last miLE logiStics'', CUP: E53D23005640006). 
It was also partially supported by the research project ``Programma ricerca di Ateneo UNICT 2024-26 NOVA - Network Optimization and Vulnerability Assessment'' of the University of Catania.

The authors are members of the Gruppo Nazionale per l'Analisi Matematica, la Probabilit\`a e le loro Applicazioni (GNAMPA - National Group for Mathematical Analysis, Probability and their Applications) of the Istituto Nazionale di Alta Matematica (INdAM - National Institute of Higher Mathematics).

\bibliography{references}
\bibliographystyle{acm}

\end{document}